\documentclass[a4paper]{article}

\usepackage{
amsmath,
amsthm,
amscd,
amssymb,
wasysym
}
\usepackage{comment}
\usepackage{tikz}
\usetikzlibrary{positioning,arrows,calc}
\usepackage{xspace}
\usepackage{stmaryrd}

\usepackage{graphicx}

\usepackage{url}
\theoremstyle{plain}
\newtheorem{theorem}{Theorem}
\newtheorem*{theorem*}{Theorem}
\newtheorem{lemma}[theorem]{Lemma}
\newtheorem{definition}[theorem]{Definition}
\newtheorem{corollary}[theorem]{Corollary}
\newtheorem{proposition}[theorem]{Proposition}
\newtheorem{conjecture}[theorem]{Conjecture}
\newtheorem{question}[theorem]{Question}
\newtheorem{remark}[theorem]{Remark}
\newtheorem{problem}[theorem]{Problem}

\newtheoremstyle{derp}% <name>
{3pt}% <Space above>
{3pt}% <Space below>
{}% <Body font>
{}% <Indent amount>
{\upshape}% <Theorem head font>
{:}% <Punctuation after theorem head>
{.5em}% <Space after theorem headi>
{}% <Theorem head spec (can be left empty, meaning `normal')>
\theoremstyle{derp}
\newtheorem{example}{Example}

\newcommand{\Z}{\mathbb{Z}}

\newcommand{\N}{\mathbb{N}}

\newcommand\xqed[1]{%
  \leavevmode\unskip\penalty9999 \hbox{}\nobreak\hfill
  \quad\hbox{#1}}
\newcommand\qee{\xqed{$\fullmoon$}}

\newcommand{\supp}{\mathrm{supp}}

\newcommand{\sph}{\mathbb{S}}

\title{Recoding Lie algebraic subshifts}

\author{
Ville Salo \and Ilkka T\"orm\"a
}

\begin{document}
\maketitle

\begin{abstract}
We study internal Lie algebras in the category of subshifts on a fixed group -- or Lie algebraic subshifts for short. We show that if the acting group is virtually polycyclic and the underlying vector space has dense homoclinic points, such subshifts can be recoded to have a cellwise Lie bracket. On the other hand there exist Lie algebraic subshifts (on any finitely-generated non-torsion group) with cellwise vector space operations whose bracket cannot be recoded to be cellwise. We also show that one-dimensional full vector shifts with cellwise vector space operations can support infinitely many compatible Lie brackets even up to automorphisms of the underlying vector shift, and we state the classification problem of such brackets.

From attempts to generalize these results to other acting groups, the following questions arise: Does every f.g. group admit a linear cellular automaton of infinite order? Which groups admit abelian group shifts whose homoclinic group is not generated by finitely many orbits? For the first question, we show that the Grigorchuk group admits such a CA, and for the second we show that the lamplighter group admits such group shifts.
\end{abstract}

\section{Introduction}

If a subshift has a group structure defined by shift-commuting continuous operations, then we can recode it so that the group operations become cellwise \cite{Ki87}. In other words, up to topological conjugacy every group subshift is a subgroup $X$ of a full shift $A^\Z$ where $A$ itself is a group and the group operation is given by $(x \cdot y)_i = x_i \cdot y_i$ for all $x, y \in A^\Z$ and $i \in \Z$. We gave another proof of this recodability in \cite{SaTo12d} based on a property of the variety (see \cite{BuSa81}) of groups: this variety is \emph{shallow}. This means that if the variety is defined by function symbols $f_i$ of arity $n_i$ for $i \in I$, then every composition of the functions $\xi \mapsto f_i(x_1, \ldots, x_{k-1}, \xi, x_{k+1}, \ldots, x_{n_i})$ with $1 \leq k \leq n_i$ and $x_j \in X$ is equivalent to such a composition of bounded length. For groups these functions are exactly multiplication by fixed group elements from the left and right, as well as inversion, and their compositions have the form $\xi \mapsto x \cdot \xi^{\pm 1} \cdot y$ for $x, y \in X$.

Shift-commuting continuous operations are the morphisms in the category of subshifts \cite{Ma71,SaTo15}, and if a subshift $X$ has such operations that put it in a variety $\mathcal{V}$, we say $X$ is an \emph{internal} algebra in this variety. We study the aforementioned property of \emph{cellwiseability}, meaning the existence of a recoding of $X$ into a subshift that sends the algebra operations into ones that are defined cellwise. Cellwising a subshift makes it much easier to deal with, as we can study algebras with cellwise operations through the finite quotient algebras obtained by restricting to words of finite length, without words getting shorter as operations are applied. The actual algebra is then the inverse limit of the algebras on finite words.

We can more generally do this in the category of $G$-subshifts for a finitely-generated group $G$. For many varieties $\mathcal{V}$ that one encounters in everyday life, the following are equivalent:
\begin{itemize}
\item for all f.g. groups $G$, all internal $\mathcal{V}$-subshifts over $G$ can be cellwised,
\item all internal $\mathcal{V}$-subshifts over $\Z$ can be cellwised,
\item the variety $\mathcal{V}$ is $k$-shallow for some $k \in \N$.
\end{itemize}
This equivalence holds for groups, monoids, semigroups, rings, vector spaces over a finite field, heaps, Boolean algebras and distributive lattices (which are all shallow), and also for quasigroups, loops\footnote{This can be obtained from Example 4 in \cite{SaTo12d} by taking the disjoint union with a singleton subshift that acts as an identity element.} and lattices (which are not shallow). An open case mentioned in \cite{SaTo12d} are the modular lattices: we do not know whether modular lattice subshifts can be recoded to be cellwise.
It is known that this variety is not $k$-shallow for any $k$, and in \cite{SaTo12d} we only give an uncellwiseable non-modular lattice example.

In this paper, we study the internal Lie algebras, i.e.\ subshifts that are in the internal variety of Lie algebras over a fixed field $K$, which we call \emph{Lie algebraic subshifts} (over $K$). The variety of Lie algebras is not shallow -- either by the characterization of free Lie algebras, or by Corollary~\ref{cor:NotCellwise} below -- and the equivalence above holds for it, i.e.\ there exist Lie algebraic $\Z$-subshifts whose Lie bracket cannot be recoded to be cellwise. The underlying vector spaces of these examples do not have dense homoclinic points, and we can produce them on any group $G$ that admits a linear cellular automaton of infinite order.

\begin{theorem*}
\label{thm:NonCellwiseableIntro}
Let $K$ be a field, $d \geq 1$ and $G$ a group such that $(K^d)^G$ admits a linear cellular automaton of infinite order. Then $G$ admits a non-cellwiseable Lie algebraic subshift.
\end{theorem*}

We do not know whether all groups admit such cellular automata, but we provide an example for the Grigorchuk group.

As our main positive result we show that all Lie algebraic subshifts where homoclinic points are dense and well-behaved can be recoded to have a cellwise bracket.

\begin{theorem*}
Let $G$ be a finitely-generated group and $X \subset A^G$ an internal Lie algebra over a finite field $K$. If the set of homoclinic points of $X$ is generated, as a vector space, by finitely many $G$-orbits, and is dense in $X$, then $X$ is cellwiseable.
\end{theorem*}

A non-trivial recoding is indeed required even in the case of a one-dimensional full shift, i.e.\ after the vector operations have been cellwised, there can still be infinitely many distinct brackets compatible with the vector shift structure. We briefly study this issue in Section~\ref{sec:Compatibles} and state the classification problem for brackets compatible with a particular vector shift structure.

If $G$ is virtually polycyclic, then the homoclinic points of $X$ are generated by finitely many orbits (essentially due to classical results of Hall), so the condition on their density is sufficient.
We show that not all groups have this property: indeed on the lamplighter group (the non-virtually polycyclic metabelian group $\Z_2 \wr \Z$) we give an example of a vector shift whose homoclinic points are not finitely orbit-generated.

If $G$ is amenable and $X$ satisfies a suitable gluing property (mean TMP), then the assumption that homoclinic points are dense is equivalent to having a trivial topological Pinsker factor. In our non-cellwiseable examples, the closure of the homoclinic points is a subgroup of finite index in $X$, and the quotient is isomorphic to the additive group of the underlying field $K$.

\section{Preliminaries}

\subsection{Symbolic dynamics on groups}

Let $G$ be a group, usually a finitely generated infinite one, and $A$ a finite alphabet.
The \emph{full $G$-shift over $A$} is the set $A^G$, whose elements are \emph{configurations}, equipped with the product topology.
It is a compact metric space where $G$ acts from the left by shifting: $(g x)_h = x_{g^{-1} h}$ for $x \in A^G$, $g, h \in G$.
A $G$-invariant closed subset $X \subset A^G$ is a \emph{$G$-subshift}.
If $Y \subset B^G$ is another $G$-subshift, a \emph{block code} is a $G$-equivariant continuous map $f : X \to Y$.
Each block code is defined by a finite neighborhood $N \subset G$ and a local function $F : A^N \to B$ by $f(x)_g = F((g^{-1} x)|_N)$, and the \emph{radius} of $f$ is the smallest $r \geq 0$ such that the $r$-ball $B_r(1_G) \subset G$ can be a neighborhood of $f$.
A bijective block code is a \emph{topological conjugacy}, and if one exists, we say $X$ and $Y$ are \emph{conjugate}.
If $X = Y$, then $f$ is a \emph{cellular automaton} (CA) on $X$.
A finite product $\prod_{i=1}^n A_i^G$ of full shifts is naturally isomorphic to the full shift $\left( \prod_{i=1}^n A_i \right)^G$ over the product alphabet, and we frequently identify them in order to talk about $n$-ary block codes of the form $f : X^n \to X$ for $X \subset A^G$.
If $n = 0$, we see $f$ as a $G$-invariant element of $X$.

A continuous function $f : X \to X$ on a subshift $X$ is called a \emph{nonuniform CA} if there exists a finite set $N \subset G$ such that for all $g \in G$, $f(x)_g$ is a function of $(g^{-1} x)|_N$ for $x \in X$, and the minimal $r \geq 0$ such that we can choose $N = B_r(1_G)$ (ball of radius $r$) is called the \emph{radius} of $f$.

\subsection{Universal algebra}

A \emph{type} of algebras is a family $\mathcal{T}$ of function symbols $(f_i)_{i \in I}$ of respective arities $(n_i)_{i \in I}$.
A \emph{term} of type $\mathcal{T}$ over a variable set $X$ is either an element $x \in X$ or an expression $f_i(t_1, \ldots, t_{n_i})$ where $f_i$ is a function symbol of arity $n_i$ in $\mathcal{T}$ and each $t_j$ is a term.
An \emph{algebra} of type $\mathcal{T}$ (or a \emph{$\mathcal{T}$-algebra}) is a set $A$ together with functions $F_i : A^{n_i} \to A$ for each function symbol $f_i$. We usually identify $F_i$ with $f_i$, and use the same symbol for both.
A function between ($\mathcal{T}$-)algebras that intertwines the respective functions is a ($\mathcal{T}$-)\emph{homomorphism}, and a bijective homomorphism is an isomorphism.

\begin{remark}
Often we say that a set (or a subshift) $X$ is a $\mathcal{T}$-algebra, without listing the operations. This typically does not leads to confusion, but when $X$ is implicitly carrying some structure, this convention makes it somewhat awkward to drop the structure. This only comes up in Example~\ref{ex:UpToIsomorphism} where we consider a Lie algebraic subshift without its underlying vector space structure.
\end{remark}

A \emph{variety} $\mathcal{V}$ of algebras of type $\mathcal{T}$ is defined by a set of \emph{identities} $(L_j \approx R_j)_{j \in J}$, where $L_j, R_j$ are terms over some abstract set of variables.
An algebra belongs to $\mathcal{V}$ if it satisfies each of its identities for all choices of values for each variable.
For example, the variety of abelian groups is defined by the function symbols $({+}, {-}, 0)$ of arities $2$, $1$ and $0$, as well as the identities $(x + y) + z \approx x + (y + z)$, $x + 0 \approx x$, $x + (-x) \approx 0$ and $x + y \approx y + x$.
A standard reference on this topic is \cite{BuSa81}.

We recall some definitions from \cite{SaTo12d}.
If $\mathcal{V}$ is a variety as above, a $\mathcal{V}$-\emph{subshift} over a group $G$ is a subshift $X \subset A^G$ equipped with block codes $f_i : X^{n_i} \to X$ that give $X$ a $\mathcal{V}$-algebra structure.
For example, group subshifts are subshifts with a group structure given by block codes, and they have been studied extensively in the literature \cite{Ki87,Sc95,BoSc08}.
We say $X$ is \emph{cellwise} if each $f_i$ admits $\{1_G\}$ as a neighborhood.
In this case $A$ is a finite $\mathcal{V}$-algebra and $X$ is a sub-algebra of the direct product $A^G$.
If there exists a cellwise $\mathcal{V}$-algebra $Y \subset B^G$ and a topological conjugacy $\phi : X \to Y$ that is also a $\mathcal{V}$-isomorphism, we say $X$ is \emph{cellwiseable}.

The \emph{affine maps} on a $\mathcal{V}$-algebra $X$ are terms over the variable set $X \cup \{\xi\}$ defined inductively as follows:
\begin{itemize}
\item
  The term $\xi$ is affine. Its depth is $0$.
\item
  If $t$ is an affine map of depth $d$, $f_i : X^{n_i} \to X$ is an $n_i$-ary algebra operation, $0 \leq j < n_i$, and $y_k$ for $0 \leq k < n$, $k \neq j$, are elements of $X$, then $f_i(y_0, \ldots, y_{k-1}, t, y_{k+1}, \ldots, y_{n-1})$ is an affine map of depth $d+1$.
\end{itemize}
Each affine map defines a function $X \to X$ by substituting $\xi$ with the function argument and evaluating the operations in $X$.
For example, if $R$ is a ring and $r, s \in R$, then the term $t = r \cdot \xi + s$ is an affine map, and defines a function $a \mapsto r a + s$ on $R$.
We often identify affine maps with the functions they define, when there is no danger of confusion.
We say $X$ is $k$-\emph{shallow} if each affine map on $X$ defines a function that is equivalent to the function of some affine map of depth at most $k$, and \emph{shallow} if it is $k$-shallow for some $k \in \N$.

If $X$ is a $\mathcal{V}$-subshift, then each affine map on $X$ is a nonuniform CA.
The following result combines Theorem~2 and Theorem~3 in \cite{SaTo12d}.\footnote{In the statement of Thoerem~2, affine maps are mistakenly referred to ``block maps'' although they are nonuniform. See Theorem 4.1.2 in \cite{Sa14} for a more careful statement.}
\begin{lemma}
\label{lem:CellwiseIffInfoCanMove}
Let $\mathcal{V}$ be a variety of algebras, $G$ be a finitely-generated group and $X \subset A^G$ a $\mathcal{V}$-subshift. Then the following are equivalent:
\begin{itemize}
\item $X$ is cellwiseable.
\item There is a uniform bound on the radius of affine maps on $X$ as non-uniform cellular automata.
\end{itemize}
If $X$ is shallow, then it is cellwiseable.
\end{lemma}

\subsection{Homoclinic points and entropy}

Let $G$ be a group that acts on a compact metrizable group $X$ by continuous automorphisms.
A point $x \in X$ is \emph{homoclinic} if $g x \to 1_X$ as $g \in G$ escapes finite subsets of $G$.
If $X \subset A^G$ is a group shift, this is equivalent to the \emph{support} $\supp(x) = \{ g \in G \;|\; x_g \neq 1_A \}$ being finite.
We denote by $\Delta_X$ the set of homoclinic points of $X$, and call $X / \overline{\Delta_X}$ the \emph{co-homoclinic factor} of $X$.
Homoclinic points play an important role in our results, as they do in the general theory of group dynamical systems: our main results depend on the nontriviality of the co-homoclinic factor.

In the rest of this section, we explain the connection between homoclinic points and the topological Pinsker factor, for amenable groups. This discussion is included as background information, and is not used in the proofs of the main results.
Recall that a countable group $G$ is \emph{amenable} if it admits a \emph{left F{\o}lner sequence}, which is a sequence $F_1, F_2, \ldots$ of finite subsets of $G$ with $\bigcup_n F_n = G$ and $|g F_n \mathbin{\triangle} F_n| / |F_n| \stackrel{n \to \infty}{\longrightarrow} 0$ for each $g \in G$, where ${\triangle}$ denotes symmetric difference. We symmetrically define \emph{right F{\o}lner sequences}.
The \emph{entropy} $h(X)$ of a topological $G$-dynamical system $X$ (a continuous action of $X$ on a compact metric space $(X, d)$) has several equivalent definitions (see~\cite[Section~9.9]{KeLi16}), of which we use the following:
\begin{itemize}
\item
  For $\epsilon > 0$ and $F \subset G$, a set $D \subset X$ is $(\epsilon, F)$-separated if for all $x \neq y \in D$ there exists $g \in F$ with $d(g x, g y) \geq \epsilon$.
  Denote by $\mathrm{sep}(F, \epsilon)$ the maximum cardinality of an $(\epsilon, F)$-separated set.
  Define $h_{\mathrm{sep}}(X, \epsilon) = \limsup_n \log \mathrm{sep}(\epsilon, F_n) / |F_n|$ and $h_{\mathrm{sep}}(X) = \sup_{\epsilon > 0} h_{\mathrm{sep}}(X, \epsilon)$ where $(F_n)_{n \in \N}$ is a F{\o}lner sequence.
\item
  For $\epsilon > 0$ and $F \subset G$, a set $D \subset X$ is $(\epsilon, F)$-spanning if for all $x \in X$ there exists $y \in D$ with $d(g x, g y) < \epsilon$ for all $g \in F$.
  Denote by $\mathrm{spn}(F, \epsilon)$ the minimum cardinality of an $(\epsilon, F)$-spanning set.
  Define $h_{\mathrm{spn}}(X, \epsilon) = \limsup_n \log \mathrm{spn}(\epsilon, F_n) / |F_n|$ and $h_{\mathrm{spn}}(X) = \sup_{\epsilon > 0} h_{\mathrm{spn}}(X, \epsilon)$ where $(F_n)_{n \in \N}$ is a F{\o}lner sequence.
\end{itemize}
Then $h_{\mathrm{sep}}(X) = h_{\mathrm{spn}}(X) = h(X)$.

The \emph{topological Pinsker factor} of a dynamical system, as defined in~\cite{BlLa93} for $\Z$-actions, is the smallest factor that has zero topological entropy.
In the case of an action of an amenable group $G$ by automorphisms of a compact group $X$ it is exactly the factor group $X / \mathrm{IE}(X)$, where $\mathrm{IE}(X)$ is the closed and $G$-invariant \emph{IE-group} of $X$ (which we do not define here).
If the group shift $\Z[G]$ is also left Noetherian, as it is for virtually polycyclic groups $G$, and $X$ is abelian (i.e.\ $G \curvearrowright X$ is an \emph{algebraic action}), the group $\Delta_X$ is dense in $\mathrm{IE}(X)$, and hence $X / \overline{\Delta_X}$ is the topological Pinsker factor.
See~\cite[Section~13]{KeLi16} for more information in the general setting, and see \cite{BoSc08} for information on Pinsker factor in the case of group shifts on abelian groups.

The characterization of the Pinsker factor as the co-homoclinic factor extends to all $G$-group shifts with an additional technical property, as we show below.

\begin{definition}
Let $G$ be a group and $X \subset A^G$ a subshift.
We say $X$ has the \emph{weak topological Markov property} (TMP) if for all finite $F \subset G$ there exists a finite set $F \subset B \subset G$ such that whenever $x, y \in X$ satisfy $x|_{B \setminus F} = y|_{B \setminus F}$, then there exists $z \in X$ with $z|_B = x|_B$ and $z|_{G \setminus F} = y|_{G \setminus F}$.
In the case that $G$ is amenable, we say $X$ has the \emph{mean TMP} if such sets $C = C_n$ exist for some right F{\o}lner sequence $F = F_n$, with $|C_n \setminus F_n| / |F_n| \to 0$ as $n \to \infty$. % in any right F{\o}lner sequence, and $|C_n \setminus F_n| / |F_n| \to 0$ as $n \to \infty$.
\end{definition}

Weak TMP was defined in the preprint~\cite{BaGoMaSi18}, where it was also proved that all group shifts on countable groups satisfy it, and mean TMP was defined by Sebasti{\'a}n Barbieri in personal communication.
Shifts of finite type on amenable groups have the mean TMP.

\begin{proposition}
Let $G$ be an amenable group and $X \subset A^G$ a group shift.
Then every zero-entropy factor of $X$ factors through $X / \overline{\Delta_X}$.
If $X$ has the mean TMP, then $X / \overline{\Delta_X}$ has zero entropy.
\end{proposition}

In particular, in the latter case the co-homoclinic factor and the topological Pinsker factor coincide.

\begin{proof}
By Lemma~\ref{lem:CellwiseIffInfoCanMove} may assume $X$ is cellwise, i.e.\ $A$ is a group and operations of $X$ are cellwise operations of $A$. We first show that every zero-entropy factor factors through $X / \overline{\Delta_X}$. For this, we show the contrapositive, that if the congruence defined by $\overline{\Delta_X}$ is not in the kernel of a factor map, then the image of the factor map has positive entropy. So let $\phi : X \to Y$ be a factor map and assume $\phi(x) \neq \phi(x')$ for some $x \in X$ and $x' \in x \cdot \overline{\Delta_X}$.
By continuity, there exists $z \in \Delta_X$ with $\phi(x) \neq \phi(x z)$.
Choose $\epsilon > 0$ small and $\supp(z) \subset B \subset G$ finite such that $d(\phi(y), \phi(y')) \geq \epsilon$ for all $y, y' \in X$ with $y|_B = x|_B$, $y'|_B = xz|_B$.

Let $\mu$ be the normalized Haar measure on $X$, and denote $E = \{ y \in X \;|\; y|_B = x|_B \}$.
Then $a =: \mu(E) > 0$, so in particular, by shift-invariance of $\mu$, for arbitrarily large $n$ there exists $y_n \in X$ with $|K_n|/|F_n| > a/2$, where $K_n = \{ g \in F_n \;|\; g y_n \in E \}$.
Choose a subset $L_n \subset K_n$ with $|L_n| \geq |K_n| / |B B^{-1}|$ and $B B^{-1} g \cap L_n = \{g\}$ for each $g \in L_n$.
For each $L \subset L_n$ we define $z^n_L = y_n \cdot \prod_{g \in L} g^{-1} z \in X$ (note that the product commutes).
We claim that $\{ \phi(z^n_L) \;|\; n \in \N, L \subset L_n \}$ forms an $(\epsilon, F_n)$-separated set in $Y$ of size at least $2^\frac{a |F_n|}{2 |B B^{-1}|}$.
Namely, if $L \neq L'$ and $g \in L \setminus L'$, then $(g z^n_L)|_B = (x z)|_B$ but $(g z^n_{L'})|_B = x|_B$, so $d(\phi(g z^n_L), \phi(g z^n_{L'})) \geq \epsilon$.
This implies $h_{\mathrm{sep}}(Y, \epsilon) \geq a \log 2 / (2 |B B^{-1}|)$ for arbitrarily small $\epsilon$, and hence positive entropy for $Y$.

For the second claim, denote $Y = X / \overline{\Delta_X}$ and choose $\epsilon > 0$.
Take $B \ni 1_G$ such that $x|_B = y|_B$ for $x, y \in X$ implies $d(\phi(x), \phi(y)) < \epsilon$. Equivalently, $d(\phi(x), \phi(y)) < \epsilon$ whenever $bx_{1_G} = by_{1_G}$ holds for all $b \in B^{-1}$.
Let $(F_n)_n$ be the right F{\o}lner sequence from the definition of mean TMP and take $C_n \subset G$ given for $F_n$ by the mean TMP. Then $F_n^{-1}$ is a left F{\o}lner sequence, and a standard calculation shows that then also $F_n' = \{ g \in G \;|\; B^{-1}g \subset F_n^{-1} \}$ is a left F{\o}lner sequence, thus it is enough to show entropy along this F{\o}lner sequence is zero.

For each pattern $P$ of shape $C_n \setminus F_n$, choose a configuration $x_P \in X$ with $x_P|_{C_n \setminus F_n} = P$.
We claim that the configurations $\phi(x_P)$ form an $(\epsilon, F_n')$-spanning set in $Y$ of size at most $|A|^{|C_n \setminus F_n|}$. This indeed suffices, as $|A|^{|C_n \setminus F_n|}$ is subexponential $|F_n|$, thus in $|F_n'|$ since $|F_n'| \geq |F_n|/2$ for large enough $n$, hence implies $h_{\mathrm{spn}}(Y, \epsilon) = 0$.
To show this set is spanning, let $y \in X$ be arbitrary.
Then some $x_P$ satisfies $x_P|_{C_n \setminus F_n} = y|_{C_n \setminus F_n}$, and by our choice of $C_n$ there exists $z \in X$ with $z|_{F_n} = y|_{F_n}$ and $z|_{G \setminus F_n} = x_P|_{G \setminus F_n}$.
We have $\phi(x) = \phi(z)$ since $x z^{-1} \in \Delta_X$, thus also $\phi(gx) = \phi(gz)$ for all $g \in G$. Since $bg \in F_n^{-1}$ for all $b \in B^{-1}$ and $g \in F_n'$, we have $bgy_{1_G} = bgz_{1_G}$ for all such $b, g$, thus
$d(\phi(g y), \phi(g x)) = d(\phi(g y), \phi(g z)) < \epsilon$ for all $g \in F_n'$, as claimed.
\end{proof}

An example of Meyerovitch \cite{Me19} shows that for subshifts without the mean TMP, over the infinite direct product $\oplus \Z_2$ one can have an abelian group shift $X$ with positive entropy and trivial $\Delta_X$. For a finitely-generated example, note that as in Proposition~\ref{prop:NotFOG}, by acting independently on cosets, one obtains an abelian group shift with the same property on the lamplighter group $\Z_2 \wr \Z$.

\subsection{Lie algebras}

Fix a field $K$, which in this article will always be finite. The variety of Lie algebras over $K$ is the variety with a scalar multiplication operation $x \mapsto a \cdot x$ (which we usually write as just $ax$) for each $a \in K$, a nullary operation $0$, a binary addition $(x, y) \mapsto x+y$, and a binary bracket operation $(x, y) \mapsto [x, y]$, such that scalar multiplication, $0$ and addition satisfy the vector space axioms (when $K$ is finite, this is a finite list of axioms), and the bracket satisfies \emph{bilinearity}
\[ [ax + by, z] = a[x, z] + b[y, z], \qquad [z, ax + by] = a[z, x] + b[z, y], \]
\emph{reflexivity} $[x,x] = 0$ and the \emph{Jacobi identity}
\[ [[x, y], z] + [[y, z], x] + [[z, x], y]  = 0. \] 
These properties imply \emph{anticommutativity} $[x, y] = -[y, x]$. %\footnote{Our results and proofs work without modification if the reflexivity axiom is replaced by anticommutativity. When $\chr K = 2$, this strictly enlarges the set of subshifts with such a bracket.}

Two vectors $u, v$ in a Lie algebra \emph{commute} if $[u, v] = 0$. An \emph{ideal} in a Lie algebra $A$ is a subalgebra $B \subset A$ satisfying $[B, A] \subset B$. A subspace of a Lie algebra $A$ is \emph{finite-dimensional} if it is finite-dimensional as a vector space.

We write the ideal generated by any set $B \subset A$ as $[B, A]$, and the vector space generated by $B \subset A$ as $\langle B \rangle$. Inductively define $[a_0, ..., a_k] = [[a_0, ..., a_{k-1}], a_k]$ for $k \geq 3$. We refer to such expressions as simply \emph{brackets}, and $k$ is the \emph{depth} of the bracket. We call $a_0$ the \emph{base} of the bracket.

\begin{lemma}
\label{lem:NeedOnlyNonCommu}
Suppose $A$ is a Lie algebra over a field $K$ generated by vectors $(e_i)_{i \in I}$. Let $B \subset A$ be arbitrary. Then 
\[ [B, A] = \langle \{ [b, e_{i_1}, ..., e_{i_k}] \;|\; b \in B, \forall j: [b, e_{i_j}] \neq 0 \} \rangle. \]
\end{lemma}

Note that directly by bilinearity of the bracket, $[B, A]$ is generated by such expressions where $e_{i_j}$ does not commute with $[b, e_{i_1}, ..., e_{i_{j-1}}]$. What requires the Jacobi identity is that none of the $e_{i_j}$ commute with the base $b$.

\begin{proof}
Let $a \in [B, A]$. By bilinearity and anticommutativity of the bracket and the definition of the ideal generated by $B$, we can write $a$ as a linear combination of brackets of the form $[b, e_{i_1}, ..., e_{i_k}]$ where $b \in B$. Thus, it is enough to make the brackets satisfy $[b, e_{i_j}] \neq 0$ for all $j$. Say a bracket is \emph{bad} if $[b, e_{i_j}] = 0$ for some $j$ and call the minimal such $j$ the \emph{bad index}.

Suppose that whenever $a$ is written as a linear combination of brackets, at least one bracket is bad. Consider then all possible ways to write $a$ as a linear combination of such brackets, and to each such expression associate the tuple $(d, \ell, p)$ where $d$ is the maximal depth among the bad brackets, $\ell$ is the number of brackets of maximal depth, $p$ is the minimal bad index among bad brackets $[b, e_{i_1}, ..., e_{i_k}]$ achieving $k = d$.

Write $a$ as a linear combination of brackets so that $(d, \ell, p)$ is lexicographically minimal (in particular $d$ is globally minimal), and consider a bad bracket  $[b, e_{i_1}, ..., e_{i_d}]$ with $[b, e_{i_p}] \neq 0$.

If $p = 1$, then $[b, e_{i_1}] = 0$ and by bilinearity we may remove the bracket entirely, making $\ell$ smaller but without increasing $d$. Suppose then $p \geq 2$. Then writing Jacobi's identity as
\[ [[x, y], z] = [x, [y, z]] + [[x, z], y], \]
setting $x = [b, e_{i_1}, ..., e_{i_{p-2}}]$, $y = e_{i_{p-1}}$, $z = e_{i_p}$ and applying to the corresponding subbracket of $[b, e_{i_1}, ..., e_{i_p}, ..., e_{i_k}]$ gives
\begin{align*}
[b, e_{i_1}, ..., e_{i_{p-2}}, e_{i_p-1}, e_{i_p}, ..., e_{i_k}] &= [b, e_{i_1}, ..., e_{i_{p-2}}, [e_{i_{p-1}}, e_{i_p}], e_{i_{p}+1} ... e_{i_k}] \\
&+ [b, e_{i_1}, ..., e_{i_{p-2}}, e_{i_p}, e_{i_{p-1}}, e_{i_p+1},...,e_{i_k}].
\end{align*}
The second term is clearly still bad, but its bad index is smaller. Thus if we replace the LHS by the RHS in the minimal expression for $a$, we obtain a contradiction since $d$ and $\ell$ are not modified, but $p$ is decreased.
\end{proof}

While the technical statement of Lemma~\ref{lem:NeedOnlyNonCommu} is what we use in our application, we note the following slightly nicer proposition that contains the essense of it (and at least in the case of $\Z$, could be used directly).

\begin{proposition}
\label{prop:EventualCommu}
Suppose $A$ is a Lie algebra over a field $K$ generated by vectors $(e_i)_{i \in I}$ such that every $e_i$ commutes with all but finitely many $e_j$. Then every finitely-generated ideal of $A$ is finite-dimensional.
\end{proposition}

\begin{proof}
Under the assumptions of the proposition, for any fixed $a \in A$, the brackets $[a', a]$ for $a' \in A$ generate a finite-dimensional space. Namely, for $a = e_i$, writing $a'$ in terms of generators and applying bilinearity, the fact all but finitely many generators commute with $e_i$ implies that we see only finitely many generators in the simplified expression. For general $a$, write $a$ in terms of generators and again apply bilinearity.

To prove the proposition, it is enough to show that $B = \{e_i \;|\; i \in F\}$ generates a finite-dimensional ideal for any finite $F \subset I$. The previous lemma gives
\[ [B, A] = \langle \{ [b, e_{i_1}, ..., e_{i_k}] \;|\; b \in B, \forall j: [b, e_{i_j}] \neq 0 \} \rangle, \]
so by the assumption $[B, A]$ is generated as a vector space by brackets involving only finitely many generators $e_i$. In particular the top element $e_{i_k}$ can take at most finitely many values, and we conclude by the previous paragraph.
\end{proof}

\section{Lie algebraic subshifts}

\subsection{Cellwiseability}

Let $X \subset A^G$ be a shift-invariant subset of a group shift. We say $X$ is \emph{orbit-generated} by a set $Y$ if $X = \langle \{gy \;|\; g \in G, y \in Y\} \rangle$ (in the algebraic sense, as a discrete group). We say $X$ is \emph{finitely orbit-generated} if it is orbit-generated by a finite set. If $X$ is a group shift, write $\Delta_X$ for the (shift-invariant) subgroup of $X$ consisting of all $x \in X$ such that $g x \rightarrow 1_X$ as $g$ escapes finite subsets of $G$.

\begin{theorem}
\label{thm:MainTechnical}
Let $G$ be a finitely-generated group and $X \subset A^G$ an internal Lie algebra over a finite field $K$ in the category of $G$-subshifts. If $X = \overline{\Delta_X}$ and $\Delta_X$ is finitely orbit-generated (as a subgroup of $(X,+)$), then the operations of $X$ can be recoded to be cellwise.
\end{theorem}

\begin{proof}
We may assume the vector operations of $X$ are cellwise: vector spaces over $K$ are shallow so Lemma~\ref{lem:CellwiseIffInfoCanMove} applies. So assume $A$ is a finite-dimensional vector space over $K$ and $X \subset A^G$ has the vector operations induced by those of $A$. By the assumption there exists a finite set $x_1, \ldots, x_k$ such that shifts of the $x_i$ generate $\Delta_X$ (in the algebraic sense), i.e. $\Delta_X = \langle g x_i \;|\; g \in G, i \in [1,k] \rangle$.

We now claim that affine maps of $X$ have bounded radius, which suffices by Lemma~\ref{lem:CellwiseIffInfoCanMove}. First for affine maps in a Lie algebra we show the normal form $\xi \mapsto a\cdot[\xi, y_1, \ldots, y_{k-1}] + y_k$ where $a \in K$ (recall that multiplication by $a \in K$ is a unary operation of the algebra). By induction, if $T$ is a term in this normal form, we need to show the same for an affine map with depth one higher. The most difficult case of the induction step is the bracket operation, and indeed
\begin{align*}
[T, y_{k+1}] &= [a\cdot[\xi, y_1, \ldots, y_{k-1}] + y_k, y_{k+1}] \\
&= a\cdot[[\xi, y_1, \ldots, y_{k-1}] + a^{-1} \cdot y_{k}, y_{k+1}] \\
&= a \cdot [[\xi, y_1, \ldots, y_{k-1}], y_{k+1}] + [y_k, y_{k+1}].
\end{align*}
is of the claimed form. %This also shows that affine maps in a Lie algebra have the form $\xi \mapsto M(\xi) + x$, where $M : X \to X$ is linear and $x \in X$ is a constant, that is, they are affine functions on the underlying vector space.

Suppose $x, y \in X$ agree in a large ball $B_r(1_G) \subset G$ around the origin with respect to the word metric $d$ of $G$. It is enough to show that if $r$ is large enough, then $t(x)_{1_G} = t(y)_{1_G}$ for any affine map $t(\xi) = a \cdot [\xi, y_1, \ldots, y_{k-1}] + y_k$. Since the bracket is bilinear, we may assume $y = y_k = 0$.

First, consider $x \in \Delta_X$. It is easy to see that $[\{x\}, \Delta_X] = [\{x\}, X] \subset \Delta_X$. It is also easy to see that $t(x)_{1_G} = 0$ holds for every choice of $k$ and $y_1, \ldots, y_{k-1}$ if and only if the ideal generated by $x$ contains an element that is nonzero at the origin of $G$. By Lemma~\ref{lem:NeedOnlyNonCommu}, we have
\begin{equation}
\label{eq:IdealFormula}
[\{x\}, \Delta_X] = \langle \{ [x, g_{m_1} x_{i_1}, \ldots, g_{m_k} x_{i_k}] \;|\; \forall j: [x, g_{m_j} x_{i_j}] \neq 0 \} \rangle.
\end{equation}

Let $F$ be the union of supports (group elements containing nonzero values) of the homoclinic generators $x_i$. If
\[ d_{\mathrm{min}}(\supp(x), \supp(g_{m_j} x_{i_j})) \geq d_{\mathrm{min}}(\supp(x), g_{m_j} F) \geq 2R \]
then $[x, g_{m_j} x_{i_j}] = 0^G$, where $R$ is the radius of the bracket as a block map, and $d_{\mathrm{min}}(K, K') = \min_{k \in K, k' \in K'} d(k, k')$. Namely, the local rule of the bracket satisfies $[P, Q] = 0$ if at least one of $P, Q \in A^{B_R(1_G)}$ is the all-$0$ pattern. Then for each $j$ in~\eqref{eq:IdealFormula}, we have $d_{\mathrm{min}}(\supp(x), \supp(g_{m_j} x_{i_j})) < 2 R$, and since $\supp(x) \cap B_r(G) = \emptyset$, this implies $\supp(g_{m_j} x_{i_j}) \cap B_R(G) = \emptyset$ if $r$ is large enough. Then the local rule of the outermost bracket in $[x, g_{m_1} x_{i_1}, ..., g_{m_k} x_{i_k}]_{1_G}$ has the form $[P, 0^{B_R(G)}] = 0$ for some pattern $P$. Thus $z_{1_G} = 0$ for all $z \in [\{x\}, \Delta_X]$.

For general $x \in X$ with $x|_{B_r(G)} = 0^{B_r(G)}$, if some $z \in [\{x\}, X]$ satisfies $z_{1_G} \neq 0$, then by continuity of the operations and $X = \overline{\Delta_X}$ we also have $z_{1_G} \neq 0$ for some $z \in [\{x\}, \Delta_X]$ and $x \in \Delta_X$.
\end{proof}

The following is essentially a classical result, see e.g.\ \cite[Theorem 1]{Ha54}. It is usually stated for ideals of the group ring $\Z[G]$, we sketch a version of the proof for subgroups of $A^G$ for a finite group $A$.

\begin{lemma}
\label{lem:VirtPoly}
Let $G$ be a virtually polycyclic group, $A$ a finite group and $Y \subset \Delta_{A^G}$ a shift-invariant subgroup of $A^G$. Then $Y$ is finitely orbit-generated. In particular, $\Delta_X$ is finitely orbit-generated for any group shift $X$ over $G$.
\end{lemma}

\begin{proof}
First, we may assume $G$ is strongly polycyclic, i.e.\ admits a subnormal series $G = G_0 \triangleright G_1 \triangleright \cdots \triangleright G_n = \{1_G\}$ with infinite cyclic factors. This is because $G$ has such a finite-index subgroup and the properties of being a shift-invariant subgroup and being finitely orbit-generated are not affected by moving between $G$ and its finite-index subgroups.
We proceed by induction on the Hirsch length $n$, assuming that $H =: G_1$ has this property.
Let $g \in G$ be such that $G = \bigcup_n g^n H$. %, and define a homomorphism $\pi : G \to \Z$ by $\pi(g^n h) = n$ for $n \in \Z, h \in H$.

Define $\Delta^+_{A^G}$ as the set of those $y \in \Delta_{A_G}$ with $\supp(y) \in \bigcup_{n \geq 0} g^n H$. This is an $H$-invariant set since each $g^n H$ is.
Let $Z = \{ y|_H \;|\; y \in Y \cap \Delta^+_{A_G}\}$, which is an $H$-invariant subgroup of $\Delta_{A^H}$. By the inductive assumption there is a finite set $F \subset Y \cap \Delta^+_{A^G}$ such that $\{ y|_H \;|\; y \in F \}$ orbit-generates $Z$.
Let $k \geq 0$ be minimal such that $\supp(y) \subset \bigcup_{i = 0}^{k-1} g^i H$ for all $y \in F$ and denote $K = \{1_G, g, g^2, \ldots, g^{k-1}\}$.

Let $y \in Y$ be arbitrary. We claim that there exists $z \in \langle G F \rangle$ with $\supp(z y) \subset g^n K H$ for some $n \in \Z$. We may assume $y \in \Delta^+_{A^G}$, shifting by some $g^n$ for $n > 0$ if this is not the case, and proceed by induction on the minimal $p \geq 0$ with $\supp(y) \subset \bigcup_{i=0}^{p-1} g^i H$. If $p \leq k$, then $\supp(y) \subset K H$ and we are done. Otherwise, observe that $y|_H \in Z$, so that some $z \in \langle F \rangle$ satisfies $(z y)|_H = 1$. Then $g^{-1} \cdot z y \in \Delta^+_{A^G}$ has a smaller value of $p$, so there exists $x \in \langle G F \rangle$ with $\supp(x (g^{-1} \cdot z y)) \subset g^n K H$ for some $n$, and hence $\supp((g \cdot x) z y) \subset g^{n+1} K H$.

Let $Z' = \{ y \in Y \;|\; \supp(y) \subset K H \}$. This is an $H$-invariant subgroup isomorphic to a subgroup of $\Delta_{B^H}$ for some finite group $B$ with $|B| = |A|^k$. Hence it is orbit-generated by a finite set $F' \subset Z'$, and then $Y$ is orbit-generated by $F \cup F'$.
\end{proof}

%Suppose $Y \leq \Delta_{A^G}$ is a subgroup. To each $y \in Y$ associate the element $s(y) = g^{-\min(\pi(\supp(y)))} y \in Y$. Observe that
%\begin{align*}
%h s(y) &= h g^{-\min(\pi(\supp(y)))} y \\
%&= g^{-\min(\pi(\supp(y)))} g^{\min(\pi(\supp(y)))} h g^{-\min(\pi(\supp(y)))} y \\
%&= g^{-\min(\pi(\supp(y)))} h'y \\
%&= g^{-\min(\pi(\supp(h'y)))} h'y \\
%&= s(h'y)
%\end{align*}
%for some $h' \in H$, so $s(Y)|_H = Z$ is a shift-invariant subgroup of $A^H$. By the inductive assumption, $s(F)|_H$ orbit-generates $Z$ for some finite $F \subset Y$.
%
%Observe that by shifting $y \in Y$ and multiplying it by suitable translates of configurations in $F$, we can turn it into a configuration $y' \in Y$ with $\supp(y) \subset KH$ for a fixed finite set $K \subset G$. Consider the set $Y' = \{ y \in Y \;|\; \supp(y) \subset K H \}$. This is an $H$-invariant subgroup of $Y$ and the action of $H$ is expansive in it. Now, apply the induction hypothesis to obtain a finite orbit-generating set $F' \subset Y'$ for $Y'$. Then $F \cup F'$ orbit-generates $Y$.
%\end{proof}

%

\begin{theorem}
\label{thm:MainIGuess}
Let $G$ be a virtually polycyclic group and suppose $X \subset A^G$ is a Lie algebraic subshift over a finite field $K$ in the category of $G$-subshifts. If the homoclinic points are dense, then the operations of $X$ can be recoded to be cellwise.
\end{theorem}

\begin{proof}
Direct from Lemma~\ref{lem:VirtPoly} and Theorem~\ref{thm:MainTechnical}.
\end{proof}

We do not know whether virtual polycyclicity is needed in the theorem.

\begin{question}
Are Lie algebraic subshifts with dense homoclinic points cellwiseable, on every group?
\end{question}

\subsection{Brackets compatible with a particular vector shift}
\label{sec:Compatibles}

After cellwiseability of the bracket, an interesting question is to try to understand what the bracket can actually do. We show some preliminary results. We concentrate on full vector shifts $V^G$. When $V$ has dimension at least $2$, there are always infinitely many compatible brackets:

\begin{proposition}
Let $V$ be a $d$-dimensional vector space over a finite field $K$, with $d \geq 2$. If $G$ is an infinite f.g. group then the full vector shift $V^G$ admits infinitely many compatible Lie brackets.
\end{proposition}

\begin{proof}
If $[] : V^2 \to V$ is any compatible Lie bracket, then $\llbracket x, y \rrbracket = f^{-1}([f(x), f(y)])$ is a compatible Lie bracket for any linear automorphism of $V^G$ (i.e.\ automorphism of the subshift and the vector space structures), by a direct calculation. Identify $V$ with $K^d$ and then $V^G$ with $(K^G)^d$. On $V$, define a bracket by $[(1, 0,0, ..., 0), (0, 1,0,0, ..., 0)] = (1,0,0,...,0)$ and extend by bilinearity and reflexivity and mapping everything else to zero. This is easily seen to define a Lie bracket on $V$: by bilinearity and reflexivity, it is enough to check the Jacobi identity on triples of distinct generators that do not commute with all other generators, and in fact there are no such triples. This extends to $V^G$ by cellwise operations.

Now, pick $f_g$ to be the partial shift
\[ f_g(x_1, x_2, ..., x_d)_h = ((x_1)_{hg}, (x_2)_h, ..., (x_d)_h), \]
which is obviously an automorphism. Then it is easy to show that the brackets $\llbracket x, y \rrbracket = f_g^{-1}([f_g(x), f_g(y)])$ are distinct for each $g \in G$.
\end{proof}

The linear automorphism groups of $(K^d)^G$ can be quite complex once $d \geq 2$. Already if $K = \Z_2$ (the two-element field), $G = \Z$ and $d = 2$, the automorphism group of the vector shift $(K^d)^G$ is a finitely-generated group containing a two-generator free group and a copy of the lamplighter group $\Z_2 \wr \Z$. (If $d = 1$, the linear automorphisms are easily seen to be shift maps.)

One may ask whether the proposition essentially produces all Lie algebra structures. We show that the construction above does not give all brackets, indeed there can exist infinitely many non-isomorphic Lie algebraic structures with the same underlying vector shifts:

\vspace*{0.2cm}

\begin{example}
\label{ex:UpToIsomorphism}
Consider a f.g.\ group $G$, let $V = K^3$ and identify $(K^3)^G$ with $(K^G)^3$ so configurations of $V^G$ are $3$-tuples $(x, y, z)$ over $K^G$. Write $e_1$ for $(\chi_{\{1_G\}}, 0^G, 0^G)$ where $\chi_B \in \{0,1\}^G$ is the characteristic function of $B$ and write $e_2$ for the corresponding (topological) generator for the second track (by \emph{tracks}, we refer to components of a Cartesian product). Now define a bracket by
\[ [e_1, e_2] = z \]
for an arbitrary (but fixed) $z \in \{0^G\} \times \{0^G\} \times \Delta_{K^G}$. Extend this using shift-commutation, continuity, reflexivity and bilinearity, mapping everything else to $0$. %If $x \cdot y$ for configurations $K^G$ denotes the cellwise product and $\phi : K^G \to K \llbracket G \rrbracket$ for the natural map from $K^G$ to infinite formal series $\sum_{g \in G} a_g g$, then one can write the resulting operation as
%\[ [(x, y, z), (x', y', z')] = \phi(x \cdot y') \cdot \phi(z) - \phi(x' \cdot y) \cdot \phi(z). \]
Bilinearity and reflexivity are then obvious. The Jacobi identity trivially holds because the bracket satisfies $[a,b,c] = 0$ for all $a, b, c \in (K^G)^3$.

Now consider this construction when $G = \Z$, and for simplicity, let us do concrete calculations with $|K| = 2$. Let $X = (K^3)^\Z$ and for all $i$ construct a Lie bracket $[]_i$ on $X$ according to the above construction, with the choice $z^i = \chi_{\{0, i\}}$ for every $i \in \Z$. It is easy to see that the brackets $[]_i$ and $[]_{-i}$ are conjugate by an automorphism of the vector shift structure of $X$ (conjugate by the partial shift by $i$ in the third component). We claim the Lie algebraic subshifts $(X, []_i)$ and $(X, []_j)$ are not isomorphic when $0 \leq i < j$. In fact, even if we ignore the vector space structure, the algebras $(X, []_i, 0)$ where we consider only the algebra operation $[]_i$ and the zero element $0 \in X^i$ as structure, are not isomorphic in the category of subshifts.

To see this, let $p_{i,n}$ the number of $n$-periodic points (points with orbit size divisible by $p$) of $X^2$  which that map to $0$ in the operation $[]_i$, where $i \geq 0$ and $n \geq 1$. The function $n \mapsto p_{i,n}$ is clearly an invariant for $(X, []_i, 0)$. Since the number of words $abcd \in \{0,1\}^4$ such that $[ae_0 + be_0', ce_0 + de_0']$ contributes $z$ on the third track is $6$, and $24 = 6 \cdot 4$, it is easy to calculate
\begin{align*}
p_{0,n} &= 24^n \\
p_{1,n} &= 24^n + 40^n \\
p_{2,2n} &= (24^n + 40^n)^2, \; p_{2,2n+1} = 24^{2n+1} + 40^{2n+1} \\
p_{3,3n} &= (24^n + 40^n)^3, \; p_{3,3n+1} = 24^{3n+1} + 40^{3n+1}, \; p_{3,3n+2} = 24^{3n+2} + 40^{3n+2} \\
\end{align*}
and in general
\[ p_{k,n} = (24^m + 40^m)^{n/m} \mbox{ where } m = n/\gcd(k,n). \]
The functions $n \mapsto p_{k,n}$ are pairwise distinct for distinct values of $k$. Indeed, if $p^h \mid k_1$ and $p^h \nmid k_2$, let so $p^k \mid k_2$ with $k$ maximal. Picking $n = p^h$ we have $m_1 = n/\gcd(k_1, n) = 1$, $m_2 = n/\gcd(k_2, n) = p^{h-k}$ and then
\[ p_{k_1,n} = (24 + 40)^{p^h} = ((24 + 40)^{p^{h-k}})^{p^k} > (24^{p^{h-k}} + 40^{p^{h-k}})^{p^k} = p_{k_2,n}. \]
Clearly $n \mapsto p_{k,n}$ is an isomorphism invariant for $(X, []_i, 0)$ (even without the vector space structure), so $(X, []_i, 0)$ and $(X, []_j, 0)$ are not isomorphic if $i \neq j$. \qee
\end{example}

We leave open three classification problems, in increasing order of generality (and presumably difficulty).

\begin{conjecture}
Let $X = K^\Z$ be the full vector shift, where $K$ is a finite field as a vector space over itself. Then $X$ is not compatible with a nontrivial (i.e.\ noncommutative) Lie algebraic subshift structure.
\end{conjecture}

\begin{problem}
Classify Lie algebraic subshift structures consistent with the cellwise vector shift structure of $V^\Z$, for $V$ a finite-dimensional vector space over a finite field $K$, up to automorphisms of $V^\Z$.
\end{problem}

\begin{problem}
Let $G$ be a f.g.\ group. Classify Lie algebraic subshift structures consistent with the cellwise vector shift structure of $V^G$, for $V$ a finite-dimensional vector space over a finite field $K$, up to automorphisms of $V^G$.
\end{problem}

Note that the previous problem includes all vector shifts over all f.g. groups also in the general sense of internal vector spaces, since vector spaces are a shallow variety.

\section{On the failure of cellwiseability}

\subsection{Nontrivial co-homoclinic factor}

If the co-homoclinic factor is nontrivial, Theorem~\ref{thm:MainIGuess} does not necessarily hold, even if $G = \Z$. A \emph{linear cellular automaton} on a vector shift is a cellular automaton (shift-invariant continuous self-map) which is linear.

\begin{theorem}
\label{thm:NotCellwise}
Suppose the full vector shift $(K^d)^G$ over a finite field $K$ admits a linear cellular automaton of infinite order. Then $G$ admits a Lie algebraic subshift with co-homoclinic factor the additive group of $K$.
\end{theorem}

\begin{proof}
Let $X = (K^d)^G \times \{a^G \;|\; a \in K\}$, and observe that $X$ is a vector shift with the cellwise vector space operations of $K^{d+1}$. Observe that the co-himoclinic factor is indeed isomorphic to the additive group of $K$, as $\overline{\Delta_X} = (K^d)^G \times \{0^G\}$ and the subshift is an SFT, thus has mean TMP. Write $e_g = \chi_{\{g\}}$. Let $f : (K^d)^G \to (K^d)^G$ be a linear cellular automaton of infinite order. To simplify notation, identify $0$ with $(0^d)^G$ on left sides of tuples, and $a \in K$ with $a^G$ on right sides. %It is standard that the action of $f$ can be identified with the natural action of a $d$-by$d$ matrix $M \in \mathcal{M}_d(K[G])$ (where $K[G]$ is the group ring of $G$ over $K$) on $(K^G)^d$ (under the natural identification of $(K^d)^G$ and $(K^G)^d$).

Let $[(x, 0), (0, a)] = (af(x), 0)$ and extend this using shift-commutation, reflexivity and bilinearity, mapping everything else to $0$. In other words, set
\begin{align*}
[(x, a), (y, b)] &= [(x, 0), (y, b)] + [(0, a), (y, b)] \\
&= [(x, 0), (y, 0)] + [(x, 0), (0, b)] + [(0, a), (y, 0)] +  [(0, a), (0, b)] \\
&= (bf(x), 0) - (af(y), 0).
\end{align*}
Not surprisingly, bilinearity and reflexivity hold. For the Jacobi identity, observe that in the presence of the other axioms, it is enough to verify it for a set of generators of the vector shift $(K^d)^G$. This reduces the problem to the calculations
\begin{align*}
&[(x, 0), (0, a), (0, b)] + [(0, a), (0, b), (x, 0)] + [(0, b), (x, 0), (0, a)] \\
= \; &[(af(x), 0), (0, b)] + 0 - [(bf(x), 0), (0, a)] \\
= \; &(baf(x), 0) - (abf(x), 0) = 0
\end{align*}
and
\begin{align*}
&[(x, 0), (y, 0), (0, b)] + [(y, 0), (0, b), (x, 0)] + [(0, b), (x, 0), (y, 0)] \\
= \; &0 + [(bf(y), 0), (x, 0)] - [(bf(x), 0), (y, 0)] = 0.
\end{align*}

The affine maps $\phi_i$ defined inductively by $\phi_0(\xi) = \xi$, $\phi_{i+1}(\xi) = [\phi_i(\xi), (0, 1)]$ do not have bounded radius, in fact $\phi_i((x, 0)) = (f^i(x), 0)$ and since $f$ has infinite order as a cellular automaton, the radii of the maps $f^i$ are not uniformly bounded.
\end{proof}

\begin{corollary}
\label{cor:NotCellwise}
Suppose $G$ is a finitely-generated group that is not torsion, i.e.\ some $g \in G$ has infinite order. Then for any finite field $K$ there is a Lie algebraic $G$-subshift over $K$ which is not cellwiseable and has co-homoclinic factor $K$.
\end{corollary}

\begin{proof}
The right shift $f(x)_h = x_{hg}$ is linear (when well-defined on an algebraic subshift, for any cellwise algebraic structure) and has infinite order.
\end{proof}

\begin{example}
On $\Z$ consider $X = \{0,1\}^\Z \times \{0^\Z,1^\Z\}$ with componentwise and cellwise vector operations induced from those of $\Z_2$. Write $e_i = \chi_{\{i\}} \times \{0^\Z\}$ where $\chi_B \in \{0,1\}^\Z$ is the characteristic function of $B$, and write $a = a^\Z$ for $a \in \{0,1\}$. Pick any finite-support configuration $y \in \{0,1\}^\Z$. Define $[e_0, (0,1)] = y$, extend by shift-commutation, reflexivity and bilinearity and set all other images to $(0, 0)$. Now $x_0 = e_0$, $x_{i+1} = [x_i, (0,1)]$ has $0$ on its second track for all $i$, and on the first track it has the $i$th image of $e_0$ in the linear cellular automaton defined by $f(e_0) = y$. For all but two choices of $y$, this orbit is infinite, and thus the Lie bracket cannot be recoded to be cellwise. \qee
\end{example}

\subsection{Which groups admit infinite-order linear CA?}

We do not know which groups admit linear cellular automata of infinite order:

\begin{question}
\label{q:ExistsLCA}
Let $G$ be a f.g. group and $V$ a finite-dimensional vector space over a finite field $K$. Does $V^G$ necessarily admit a linear cellular automaton of infinite order?
\end{question}

If the answer to Question~\ref{q:ExistsLCA} is positive for every group, then every f.g.\ group admits a non-cellwiseable Lie algebraic subshift. If $G$ has a non-torsion element $g$, then the right shift $f_g$ defined by $f_g(x)_h = x_{hg}$ is linear and of infinite order, so the question is only interesting for torsion groups. Of course, every f.g. group $G$ admits an infinite f.g. group of linear cellular automata on $V^G$ (these right shifts), but the bracket construction in the proof of Theorem~\ref{thm:NotCellwise} only extends directly to commutative linear cellular automata actions.

Linear cellular automata can be seen as elements of the group ring $K[G]$. Recall Kaplansky's conjectures on group rings of torsion-free groups $G$ over a field $K$: the unit conjecture states that $K[G]$ has no units other than $k g$ for $k \in K^{\times}$ and $g \in G$; the zero-divisor conjecture states that $K[G]$ has no zero divisors other than $0$; the idempotent conjecture states that $K[G]$ has no idempotents other than $0$ and $1$. See~\cite{Ba14} for an overview of these. As a variant of Kaplansky's conjectures, one can conjecture that $K[G]$ contains no element $p$ with $p^n = p^m$ for any $0 \leq n < m$. This implies the idempotent conjecture and is implied by the unit conjecture. Question~\ref{q:ExistsLCA} is, in some sense, the dual of this conjecture on torsion groups, since in the case $V = K$ it asks whether the condition holds for \emph{all} $p \in K[G]$. We do not know the answer to Question~\ref{q:ExistsLCA} in general, but for the Grigorchuk group it is positive.

\begin{example}
Let $G = \langle a, b, c, d \rangle$ be the first Grigorchuk group and $K$ a finite field.
We show that $K[G]$ has an element of infinite order.

Take $p = a d a + d a d + c \in K[G]$ and denote $T = \{a d a, d a d, c\} \subset G$.
We claim $p$ has infinite multiplicative order.
For this, recall the form of the Schreier graphs of the Grigorchuk group for its natural action on $\{0,1\}^\N$, with the commonly used generators $a, b, c, d$.
The Schreier graph $S$ defined by the action of $G$ on the orbit of $111...$ is the following infinite $4$-regular multigraph $S$ (see e.g. \cite{Bo15}):
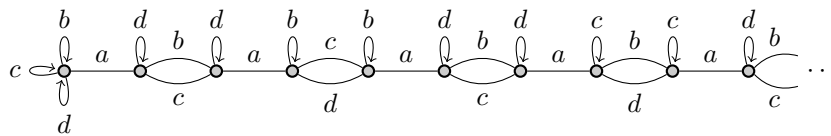
\begin{center}
\begin{tikzpicture}[
place/.style={circle,draw,thick,fill=black!20,
             inner sep=0pt,minimum size=1.5mm}]
\foreach \i in {0,1,...,4} {
	\node[place] (m\i) at (2*\i,0) {};
	\node[place] (n\i) at (2*\i+1,0) {};
}
\foreach \i/\l in {0,1,2,3,4} {
 	\draw (m\i) -- node [above] {$a$} (n\i);
}
\draw (m0) edge [loop above] node [above] {$b$} ();
\foreach \i/\j/\l/\u/\d in {0/1/d/b/c,1/2/b/c/d,2/3/d/b/c,3/4/c/b/d} {
	\draw (n\i) to[bend right=30] node [below] {$\d$} (m\j);
	\draw (n\i) to[bend left=30] node [above] {$\u$} (m\j);
 	\draw (n\i) edge [loop above] node [above] {$\l$} ();
 	\draw (m\j) edge [loop above] node [above] {$\l$} ();
}
\draw (n4) edge [loop above] node [above] {$d$} ();
\node (a) at (10,0) {$\cdots$};
	\draw (n4) to[bend right=30] node [below] {$c$} (a);
	\draw (n4) to[bend left=30] node [above] {$b$} (a);
\draw (m0) edge [loop left] node [left] {$c$} ();
\draw (m0) edge [loop below] node [below] {$d$} ();
\end{tikzpicture}
\end{center}
We have a group action $V(S) \curvearrowleft G$ of $G$ on the vertices of $S$ which just follows the edge with the respective label.

Apart from the two leftmost ones, the vertices of $S$ can be partitioned into induced subgraphs that are isomorphic to one of the following, connected at their endpoints by double edges with labels $b$ and $c$:
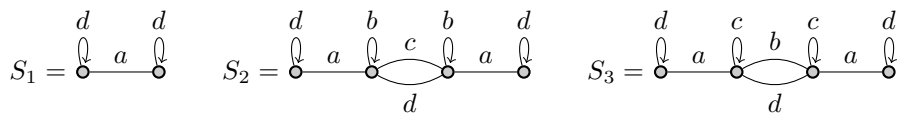
\begin{center}
\begin{tikzpicture}[
place/.style={circle,draw,thick,fill=black!20,
             inner sep=0pt,minimum size=1.5mm}]

\foreach \i in {0,1} {
	\node[place] (n\i) at (\i,0) {};
	\draw (n\i) edge [loop above] node [above] {$d$} ();
}
\draw (n0) -- node [above] {$a$} (n1);
\node [left] at (n0) {$S_1 = {}$};

\foreach \i/\l in {0/d,1/b,2/b,3/d} {
	\node[place] (k\i) at (\i+2.8,0) {};
	\draw (k\i) edge [loop above] node [above] {$\l$} ();
}
\draw (k0) -- node [above] {$a$} (k1);
\draw (k2) -- node [above] {$a$} (k3);
\draw (k1) to[bend right=30] node [below] {$d$} (k2);
\draw (k1) to[bend left=30] node [above] {$c$} (k2);
\node [left] at (k0) {$S_2 = {}$};

\foreach \i/\l in {0/d,1/c,2/c,3/d} {
	\node[place] (h\i) at (\i+7.6,0) {};
	\draw (h\i) edge [loop above] node [above] {$\l$} ();
}
\draw (h0) -- node [above] {$a$} (h1);
\draw (h2) -- node [above] {$a$} (h3);
\draw (h1) to[bend right=30] node [below] {$d$} (h2);
\draw (h1) to[bend left=30] node [above] {$b$} (h2);
\node [left] at (h0) {$S_3 = {}$};

\end{tikzpicture}
\end{center}
Take a geodesic path $\gamma_v$ of some length $n_v$ along the generator set $T$ to some vertex $v$ that is the rightmost vertex of a copy of some $S_i$.
We claim that $\gamma_v$ is the unique geodesic path to $v$ in $S$ along $T$.
The path must begin with $d a d, c$ in order to reach the leftmost copy of $S_2$.
After this, the only way to cross from one $S_i$ to the next $S_j$ is to use $c$, the only way to cross an $S_1$ is $d a d$, and the unique shortest way to cross an $S_2$ or $S_3$ is $a d a$.
Since $G$ acts on $S$, every geodesic path to the element $g_v \in G$ represented by $\gamma_v$ must form a geodesic path to $v$ in $S$ as well, hence $\gamma_v$ is also the unique geodesic path in $G$ to $g_v$ along $T$.
The coefficient of $p^{n_v}$ at $g \in G$ is precisely the number of geodesics to $g$ in $G$ along $T$ modulo $\mathrm{char}(K)$, hence for $g_v$ it equals $1_K$.
Since we also have $n_v \neq n_w$ -- and thus $g_v \neq g_w$ -- for $v \neq w$, it follows that $p$ has infinite multiplicative order in $K[G]$.

A similar proof shows that $p = a + b + c + d$ has infinite multiplicative order whenever $q = \mathrm{char}(K) \neq 2$: the number of geodesics to the vertex at distance $2 n$ from the left end of $S$ is $2^n$, which is not divisible by $q$, so at least one of the respective elements of $G$ is reached by a number of geodesics that is also not divisible by $q$.
The same strategy cannot be used when $q = 2$, but we conjecture that $p$ has infinite order in this case as well.

A similar analysis goes through for all the Grigorchuk groups $G_\omega$, in each case using one of the polynomials $p \in \{a d a + d a d + c, a c a + c a c + b, a b a + b a b + d\}$.
\qee
\end{example}

\subsection{Infinitely orbit-generated homoclinics}

If $G$ is f.g. and such that for every group shift (or even just vector shift) on $X$ has finitely orbit-generated $\Delta_X$, then all Lie algebraic subshifts on $G$ with trivial co-homoclinic factor are cellwiseable by Theorem~\ref{thm:MainTechnical}. We do not know which groups $G$ have this property. Certainly not all groups do:

\begin{proposition}
\label{prop:NotFOG}
On $\oplus_\N \Z_2$ and $\Z_2 \wr \Z$ there exist vector shifts $X$ over the two-element field, such that $\Delta_X$ is not finitely orbit-generated.
\end{proposition}

\begin{proof}
It is enough to show that this is true for $H = \oplus_\N \Z_2$ (the direct sum of infinitely many copies of $\Z_2$), since it is a subgroup of $\Z_2 \wr \Z$. The left-regular shift action of $H$ on the infinite direct product $\Z_2^H$ can be interpreted as follows: Identify $H$ with $\N$ by identifying elements of $H$ with binary expansions of numbers, so we can write elements of the direct product $\Z_2^H$ as infinite binary words. Then $H = \langle \{2^i \;|\; i \in \N\} \rangle$ and the group operation is bitwise XOR. The action of a generator $2^k$ on $\Z_2^H$ is to swap the contents of intervals
\[ [2^{k+1} n, 2^{k+1} n + 2^k) \leftrightarrow [2^{k+1} n + 2^k, 2^{k+1}(n+1)) \]
where $k, n \geq 0$. For convenience, we define also a right shift map on $\Z_2^H = \Z_2^\N$ by $s(x)_i = x_{i-1}$ (with $x_{-i} = 0$ for all $i > 0$). Define the subgroups $H_n = [0, 2^n)$ for $n \geq 0$.

Define $m_i = 4^i$ and define a sequence of binary words inductively by $u_0 = 1$ and $u_n = 0^{m_{n-1}}0^{m_{n-1}}u_{n-1}u_{n-1}$ for $n \geq 1$. The word $u_i$ is of length $m_i$ and the word $v_i = u_i^4$ is of length $m_{i+1}$. We identify a finite word $v$ with the infinite word $v0^\N \in \Z_2^H$. Define now $X \subset \Z_2^H$ as the topological closure of the subgroup generated by the configurations $s^{k m_{i+1}}(v_i)$ for $k \geq 0, i \geq 0$.

We now show the following claims:
\begin{itemize}
\item $X$ is a group shift over $H$, i.e. the translation of $H$ is well-defined on it,
\item $X$ has no ``new'' homoclinics, i.e.\ $\Delta_X = \langle \{s^{nm_{i+1}}(v_i) \;|\; i, n \in \N \} \rangle$, and
\item $X$ is not finitely orbit-generated.
\end{itemize}

For the first claim, we show the stronger fact that the subgroup $X_n = \langle \{s^{k m_{i+1}}(v_i) \;|\; i \leq n, k \in \N \} \rangle$ is closed under the action of $H$ for each $n \in \N$. For this we show by induction that every $H$-translate of $v_n$ with support contained in $[0,m_{n+1})$, in other words its every $H_{2 n + 2}$-translate, is generated by $v_n$ together with $H$-translates of $v_i$ for $i < n$. This is clearly true for $v_0 = 1 1 1 1$ which has no such nontrivial translates.

Consider then
\[ v_n = u_n^4 = (0^{m_{n-1}}0^{m_{n-1}}u_{n-1}u_{n-1})^4. \]
Every $H_{2 n + 2}$-translate of $v_n$ is obtained by first permuting the four $u_n$-blocks by the Klein four-group $\langle 2^{2 n}, 2^{2 n + 1} \rangle$ (which fixes $v_n$, thus is useless), and then applying $H_{2 n}$-translations which permute the four $u_n$-blocks separately. On the other hand every $H_{2 n}$-translate of the blocks $0^{m_{n-1}}0^{m_{n-1}}u_{n-1}u_{n-1}$ can be realized by adding vectors from $X_{n-1}$: adding $v_{n-1} = u_{n-1}u_{n-1}u_{n-1}u_{n-1}$ realizes the translation that swaps the largest intervals, and any $H_{2 n - 2}$-translation inside a $u_{n-1}$-block is obtained by adding vectors from $X_{n-2}$, by induction. This concludes the claim that $X$ is a group shift.

Observe that $\dim (X_{n-1}|_{[0,m_{n+1})}) = 4 \cdot \dim (X_{n-1}|_{[0,m_n)})$ since $X_{n-1}$ is generated by translates of vectors with support contained in $[0, m_n)$. Next, we show an auxiliary claim that $\dim (X_n|_{[0,m_{n+1})}) = \dim (X_{n-1}|_{[0,m_{n+1})}) + 1$. The upper bound follows from the fact that $X_n|_{[0,m_{n+1})} = X_{n-1}|_{[0,m_{n+1})} + \langle v_n \rangle$. We prove the lower bound by induction on $n$. For $n = 1$ this is clear since no $H$-translate of $v_0 = 1 1 1 1$ contains a subword $0 0 1 1$ aligned at a multiple of $4$. Consider now $n > 1$, and suppose for a contradiction that $v_n \in X_{n-1}$. Since $v_n$ is composed of four $0^{m_{n-1}}0^{m_{n-1}}u_{n-1}u_{n-1}$-blocks, all elements of $\langle s^{k m_n}(v_{n-1}) \;|\; k \geq 0 \rangle + v_n$ consist of blocks of the form $a a b b$ with $a, b \in \{ 0^{m_{n-1}}, u_{n-1} \}$. Thus $u_{n-1} \in X_{n-2}$, and since $X_{n-2}$ is $H$-invariant, this implies $v_{n-2} \in X_{n-2}$, a contradiction. This concludes the claim.

In particular, $X_{\infty} = \bigcup_n X_n$ is not finitely-generated. If the closure $X = \overline{X_{\infty}}$ does not contain any new homoclinic points, i.e. $\Delta_X = X_{\infty}$, then the last two items follow. For this, we show that if $x \in X_{\infty}$ and $x|_{[0,m_{i+1})} \notin X_i$ then $x_{[m_{i+1},m_{i+2})}$ is nonzero. The result $\Delta_X = X_{\infty}$ follows from this, since if $x \in X \setminus X_i$ and $\supp(x) \subset [0,m_{i+1})$, then we can find an approximation $y \in X_{\infty}$ such that $y|_{[0,m_{i+2})} = x|_{[0,m_{i+2})}$, and since $x_{[m_{i+1},m_{i+2})} = 0^{m_{i+2}-m_{i+1}}$ we necessarily have $y \in X_i$, so also $x \in X_i \subset X_{\infty}$.

To show the claim, suppose $x \in X_{\infty}$ and $x|_{[0,m_{i+1})} \notin X_i$. Observe that $X_{i+n}|_{[0,m_{i+1})} = X_{i+1}|_{[0,m_{i+1})}$ for all $n \geq 1$, and if $x|_{[0,m_{i+1})} \in X_{i+1} \setminus X_i$ then we have used a nontrivial shift of $v_{i+1}$ to produce $x$, and as seen above we cannot cancel any of the four $m_{i+1}$-blocks $u_{i+1}$ with vectors from $X_i$, so $x_{[m_{i+1},m_{i+2})}$ is nonzero as claimed.
\end{proof}

Note that the example above only shows that the proof of Theorem~\ref{thm:MainTechnical} does not extend to $\Z_2 \wr \Z$. We do not know whether all Lie algebraic subshifts on $\Z_2 \wr \Z$ which have trivial co-homoclinic factor are cellwiseable.

One may wonder if for Proposition~\ref{prop:NotFOG} it is sufficient that that $\oplus_\N \Z_2$ is infinitely-generated, i.e.\ $\Z_2 \wr \Z$ does not have the maximal condition on subgroups. In \cite{Sa18} it was shown that all such groups admit non-SFT group shifts. The above construction does not seem to directly adapt to all such groups. As a concrete example, we do not know whether $\Z[1/2] \rtimes \Z$ admits group shifts $X$ such that $\Delta_X$ is not finitely orbit-generated. Here, $\Z[1/2]$ is the dyadic rationals under addition and $\Z$ acts by multiplication by $2$.

\bibliographystyle{plain}
\bibliography{../../../bib/bib}{}

\end{document}